\documentclass[11pt,a4paper]{article}
\usepackage{amsmath,amssymb,amsthm,amscd,mathrsfs}
\usepackage{indentfirst}
\usepackage[top=25mm, bottom=30mm, left=30mm, right=30mm]{geometry}
\newtheorem{Def}{\bf Definition}[subsection]
\newtheorem{Thm}[Def]{\bf Theorem}
\newtheorem{Lem}[Def]{\bf Lemma}
\newtheorem{Cor}[Def]{\bf Corollary}
\newtheorem{Pro}[Def]{\bf Proposition}

\newtheorem{ThmA}{\bf Theorem}

\title{\bf Examples of factors which have no Cartan subalgebras}
\author{Yusuke Isono\thanks{Department of Mathematical Sciences,
University of Tokyo, Komaba, Tokyo, 153-8914, Japan \protect \\  E-mail: \texttt{isono@ms.u-tokyo.ac.jp}}}
\date{}
\begin{document}
\maketitle

\begin{abstract}
We consider some conditions similar to Ozawa's condition (AO), and prove that if a non-injective factor satisfies such a condition and has the $\rm W^*CBAP$, then it has no Cartan subalgebras. As a corollary, we prove that $\rm II_1$ factors of universal orthogonal and unitary discrete quantum groups have no Cartan subalgebras. We also prove that continuous cores of type $\rm III_1$ factors with such a condition are semisolid as a $\rm II_\infty$ factor.
\end{abstract}

\section{\bf Introduction}

In the von Neumann algebra theory, the Cartan subalgebras give us many important information and fascinating examples. 
In fact, Cartan subalgebras always come from some orbit equivalence classes in the following sense: for a given separable factor $M$ and its Cartan subalgebra $A\subset M$, there exists the unique orbit equivalence class $\cal R$ (and the cocycle $\sigma$) on a standard space $X$ such that $(L^\infty(X)\subset L({\cal R},\sigma))\simeq (A\subset M)$ $\cite{FM77}$. This correspondence sometimes enables us to make use of the ergodic theory to analyze such class of factors. This is one of the main reasons why Cartan subalgebras have been studied for a long time.

For example, Sorin Popa gave first examples of $\rm II_1$ factors whose fundamental groups are trivial $\cite{Po 1}$. In the proof, he identified fundamental groups of these factors as that of obit equivalence classes, by some deformation/intertwining arguments between two Cartan subalgebras. Hence he essentially investigated their Cartan subalgebras. This is the first result of the rigidity theory of $\rm II_1$ factors.

From this pioneering work, there has been many remarkable works:\ realization of many outer automorphism groups and fundamental groups;\ new examples of prime factors;\ uniqueness and non-existence of Cartan subalgebras;\ $\rm W^*$-superrigidity and so on. In this paper, we concentrate our attention on a negative type result, that is, non-existence of Cartan subalgebras.

Here we recall the definition of Cartan subalgebras. Let $M$ be a von Neumann algebra and $A$ an abelian subalgebra of $M$. We say $A$ is a \textit{Cartan subalgebra} of $M$ if it satisfies the following conditions:
\begin{itemize}
\item there exists a faithful normal conditional expectation from $M$ onto $A$; 
\item $A$ is maximal abelian in $M$, that is, $A'\cap M=A$;
\item the normalizer group ${\cal N}_M(A)$ generates $M$, that is, ${\cal N}_M(A)''=M$.
\end{itemize}
Here the normalizer group is defined as ${\cal N}_M(A):=\{u\in {\cal U}(M)\mid uAu^*=A\}$. 
We note that Cartan subalgebras of $M$ are diffuse (i.e.\ which have no minimal projections) if so is $M$.

Historically, examples of von Neumann algebras which have no Cartan subalgebras were first discovered by Voiculescu $\cite{Vo96}$. He showed that the free group factors $L\mathbb{F}_n$ ($n\geq2$) have no Cartan subalgebras and his method relied on the free probability theory. 
Jung generalized this result to some free product von Neumann algebras $\cite{Ju05}$. Shlyakhtenko proved that free Araki--Woods factors of type ${\rm III}_\lambda$ have no Cartan subalgebras $\cite{Sh00}$.

In the rigidity theory, Ozawa and Popa gave first examples $\cite{OP 07}$. They proved that the free group factors are such examples, and they actually proved that these factors are strongly solid. 
Here we recall that a finite von Neumann algebra $M$ is \textit{strongly solid} if for any diffuse injective subalgebra $A\subset M$, the normalizer ${\cal N}_M(A)$ generates an injective von Neumann algebra, that is, ${\cal N}_M(A)''$ is injective. It is easy to see that if a finite von Neumann algebra $M$ is strongly solid, then any non-injective diffuse von Neumann subalgebra of $M$ has no Cartan subalgebras. Hence their result is stronger than that of Voiculescu.

After the work of Ozawa and Popa, there has been many non-existence results, and in the present paper we follow $\cite{PV12}$, in which Popa and Vaes proved remarkable uniqueness and non-existence results of Cartan subalgebras. In the same paper they gave a new proof of the fact that factors of weakly amenable and bi-exact groups are strongly solid (this was first proved by Chifan, Sinclair, and Udrea $\cite[\rm Corollary\ 0.2]{CSU11}$ with an equivalent notion of bi-exactness called \textit{array}, see $\cite{CS}$ and $\cite[\rm Proposition\ 2.1]{CSU11}$). We refer to this new proof. In fact, we will prove the same statement for more general von Neumann algebras which are not necessarily group von Neumann algebras.

For this purpose, we need notions of general von Neumann algebras which correspond weak amenability and bi-exactness. It is known that weak amenability has such a notion called the $\rm W^*CBAP$ (see Subsection $\ref{wk am}$), but bi-exactness does not. Ozawa's condition (AO) is a candidate but  this is not enough for us. We will investigate it in Section $\ref{sim ao}$. See $\cite{Sh04}$, $\cite{GJ07}$, and $\cite{Is}$ for other notions similar to condition (AO).

After this consideration, we prove the following main theorems.

\begin{ThmA}[{Theorem \ref{thmA}}]\label{A}
Let $M$ be a $\rm II_1$ factor with separable predual. If $M$ satisfies condition $\rm (AO)^+$ (see Definition $\ref{ao+}$) and  has the $\it W^*CBAP$, then $M$ is strongly solid.
\end{ThmA}

\begin{ThmA}[Theorem \ref{thmB}]\label{B}
Let $M$ be a non-injective type $\rm III$ factor with separable predual and $\phi$ a faithful normal state on $M$. If $(M,\phi)$ satisfies condition $\rm (AOC)^{+}$ (see Definition $\ref{aoc+}$) and has the $\it W^*CBAP$, then $M$ has no {$\phi$-Cartan subalgebras}.
\end{ThmA}
{Here} \textit{$\phi$-Cartan subalgebra} means a Cartan subalgebra which has a $\phi$-preserving faithful normal conditional expectation $E$, that is, $\phi=\phi\circ E$.

To prove Theorem \ref{A}, we need only slight modifications of the proof of (a special case of) $\cite[\rm Theorem\ 3.1]{PV12}$. Theorem \ref{B} can be proved by seeing its continuous core, and this idea comes from $\cite{HR}$ and $\cite{HV12}$. Since condition $\rm (AOC)^+$ is similar to condition (AO) with respect to the continuous core, we naturally deduce the following primeness result. In the theorem below, Tr means the canonical semifinite trace on the continuous core.


\begin{ThmA}\label{C}
Let $M$ be a von Neumann algebra with separable predual and $\phi$ a faithful normal state on $M$. Let $\cal M$ be its continuous core with respect to $\phi$ and $p$ a projection in $\cal M$ such that $\mathrm{Tr}(p)<\infty$. If $(M,\phi)$ satisfies condition $\rm (AOC)^{+}$, then $p{\cal M}p$ is semisolid. In particular, $\cal M$ is a semisolid type $\rm II_\infty$ factor if $M$ is a type $\rm III_1$ factor.
\end{ThmA}

Here we recall that a finite von Neumann algebra $M$ is \textit{semisolid} (respectively, \textit{solid}) if for any type $\rm II$ (respectively, diffuse) subalgebra $N\subset M$, the relative commutant $N'\cap M$ is injective. For a semifinite von Neumann algebra $M$, semisolidity (and solidity) is defined as that of $pMp$ for all finite projections $p\in M$. We also recall that $M$ is \textit{prime} if for any tensor decomposition $M=M_1\otimes M_2$, one of $M_i$ ($i=1,2$) is of type I. It is not difficult to see that semisolidity implies primeness for non-injective semifinite factors. Hence the conclusion of the theorem above implies primness.

The aim of our generalization is, of course, to find new examples. Factors of universal discrete quantum groups $A_o(F)$ and $A_u(F)$ (see Subsection $\ref{uni dis}$) are our main targets. On the one hand, it is known that they satisfy condition (AO) $\cite{VVe}\cite{VV08}$, and we will observe that they in fact satisfy a little stronger conditions. See Section $\ref{sim ao}$ for the details. On the other hand, weak amenability of them were shown very recently $\cite{Fr12}$ but only for the case that they are monoidally equivalent to $A_o(1_n)$ or $A_u(1_n)$. Thus combined with the main theorems, we have the following corollary.

\bigskip
\noindent
{\bf Corollary.} \textit{Let $\mathbb{G}$ be a universal discrete quantum group $A_o(F)$ or $A_u(F)$ for $F\in \mathrm{GL}(n,\mathbb{C})$ $(n\geq3)$. Denote the Haar state by $h$.}
\begin{itemize}
\item[$(1)$] \textit{If $F=1_n$, then $L^\infty(\mathbb{G})$ is strongly solid. In particular, $L^\infty(\mathbb{G})$ has no Cartan subalgebras.}
\item[$(2)$] \textit{If $L^\infty(\mathbb{G})$ is non-injective and has the $W^*$CBAP, then it has no $h$-Cartan subalgebras.}
\item[$(3)$] \textit{If $L^\infty(\mathbb{G})$ is a non-injective type $\rm III_1$ factor, then the continuous core $L^\infty(\mathbb{G})\rtimes_{\sigma^h}\mathbb{R}$ is a semisolid, in particular prime, $\rm II_\infty$ factor.}
\end{itemize}

We will observe in Subsection $\ref{so}$ that the continuous core of $L^\infty(A_o(F))$ is semisolid but never solid for some concrete matrix $F$.

Theorem \ref{B} works for the case that $F$ is not an identity matrix, but we do not know whether $L^\infty(\mathbb{G})$ has the $\rm W^*$CBAP or not  for a general matrix $F$. If one obtains this property, Theorem \ref{B} is applicable for every non-injective von Neumann algebras of $A_o(F)$ and $A_u(F)$, and hence one has non-existence results for them. We leave this problem as follows.


\bigskip
\noindent
{\bf Problem.} \textit{When do von Neumann algebras of universal discrete quantum groups $A_o(F)$ and $A_u(F)$ have the $W^*$CBAP?}

\section{\bf Preliminaries}




\subsection{\bf Tomita--Takesaki theory}\label{TT}

For Tomita--Takesaki theory, we refer the reader to $\cite{Tak 2}$.

Let $M$ be a von Neumann algebra and $\phi$ a faithful normal state on $M$. We first consider the following anti-linear map:
\begin{equation*}
S\colon M\Omega \rightarrow M\Omega\subset L^2(M,\phi) ; \ a\Omega \mapsto a^*\Omega,
\end{equation*}
where $\Omega$ is the canonical cyclic separating vector associated with $\phi$. This map is closable on $L^2(M,\phi)$ and write as $S=J\Delta^{1/2}$ the polar decomposition of $S$. We call $\Delta$ the \textit{modular operater} and $J$ the \textit{modular conjugation}. The following fundamental relations are important:
\begin{equation*}
JMJ=M',\quad \Delta^{it}M\Delta^{-it}=M \quad(t\in\mathbb{R}).
\end{equation*}
In the paper, we frequently identify $JMJ$ as the opposite algebra $M^{\rm op}$ with the obvious correspondence.
The GNS-representation on the Hilbert space $L^2(M,\phi)$ (with a faithful normal state $\phi$) is called a \textit{standard representation} (see $\cite[\rm Definition\ IX.1.14]{Tak 2}$ for the intrinsic definition).

From the relation above, $\sigma_t^{\phi}(a):=\Delta^{it}a\Delta^{-it}$ $(a\in M, t\in \mathbb{R})$ defines a one parameter automorphism group on $M$, which is called the \textit{modular automorphism group} on $M$ associated with $\phi$. The \textit{continuous core} of $M$ is defined as the crossed product von Neumann algebra $\widetilde{M}:=M\rtimes_{\sigma^\phi} \mathbb{R}$ and it does not depend on the choice of $\phi$. We can then construct a semifinite weight $\tilde{\phi}$ on the core called the \textit{dual weight of} $\phi$ $\cite[\rm Definition\ X.1.16]{Tak 2}$, which of course depends on $\phi$. 
The dual weights are always faithful and its modular action becomes inner (more precisely $\sigma_t^{\tilde{\phi}}=1\otimes \lambda_t$) so that $\widetilde{M}$ is always semifinite. A canonical semifinite trace on $\widetilde{M}$ is given by $\mathrm{Tr}:=\tilde{\phi}(h\cdot )$, where $h$ is the self-adjoint map satisfying $h^{it}=1\otimes \lambda_t$ $(t\in \mathbb{R})$.
We say a type $\rm III$ factor $M$ is of type $\rm III_1$ if the continuous core is a factor of type $\rm II_\infty$ (see $\cite[\rm Definition\ XII.1.3]{Tak 2}$ for definitions of type $\rm III_\lambda$). 

The associated representations
\begin{alignat*}{5}
\pi&\colon& M\rightarrow \mathbb{B}(L^2(M)\otimes L^2(\mathbb{R})) &;&\ &x\mapsto\int_\mathbb{R}\sigma_{-t}^\phi(x)\otimes e_t\cdot dt ,\\
u&\colon& \mathbb{R}\rightarrow {\cal U}(L^2(M)\otimes L^2(\mathbb{R})) &;&\ &t\mapsto 1\otimes \lambda_t,&
\end{alignat*}
where $(\int_\mathbb{R}\sigma_{-t}^\phi(x)\otimes e_t\cdot dt\xi)(s):=\sigma_{-s}^\phi(x)\xi(s)$ and $((1\otimes\lambda_t)\xi)(s):=\xi(-t+s)$ for any $\xi\in L^2(M)\otimes L^2(\mathbb{R})$, give a standard representation of $M\rtimes_{\sigma^\phi} \mathbb{R}$ on $L^2(M)\otimes L^2(\mathbb{R})$ with respect to the dual weight $\tilde{\phi}$. 

A conditional expectation from $\widetilde{M}$ onto $L\mathbb{R}$ is defined by $E_{L\mathbb{R}}(x\lambda_t):=\phi(x)\lambda_t$ ($x\in M, t\in\mathbb{R}$). Then $E_{L\mathbb{R}}$ is $\tilde{\phi}$-preserving and Tr-preserving.
The modular conjugation $\tilde{J}$ on $\widetilde{M}$ is given by 
\begin{equation*}
(\tilde{J}\xi)(t):=\Delta^{-it}J\xi(-t)\quad (t\in \mathbb{R}, \xi\in L^2(M)\otimes L^2(\mathbb{R}))
\end{equation*}
$\cite[\rm Lemma\ X.1.13]{Tak 2}$, and one can easily verify that
\begin{eqnarray*}
\tilde{J}\pi(x)\tilde{J}=JxJ\otimes 1\quad (x\in M), \quad \tilde{J}(1\otimes \lambda_t)\tilde{J}=\Delta^{it}\otimes\rho_t\quad(t\in\mathbb{R}),
\end{eqnarray*}
where $\rho_t$ is the right translation defined by $(\rho_t\eta)(s):=\eta(s+t)$ ($\eta\in L^2(\mathbb{R})$). Hence we have
\begin{eqnarray*}
(M\rtimes_{\sigma^\phi} \mathbb{R})'=\tilde{J}(M\rtimes_{\sigma^\phi} \mathbb{R})\tilde{J}&=&W^*\{JxJ\otimes 1\ (x\in M),\ \Delta^{it}\otimes\rho_t\ (t\in\mathbb{R})\}\\
&=&W^*\{M'\otimes 1,\ \Delta^{it}\otimes\rho_t\ (t\in\mathbb{R})\}.
\end{eqnarray*}

Next, we investigate how Cartan subalgebras of $M$ behave in the continuous core of $M$. Let $M$ be a general von Neumann algebra, $A\subset M$ a Cartan subalgebra of $M$, and let $E_A$ be an associated conditional expectation. Take a faithful normal state $\phi$ on $A$ and extend it on $M$ via $E_A$ (still denote it by $\phi$). 
Then by the proof of Takesaki's conditional expectation theorem $\cite[\rm Theorem\ IX.4.2]{Tak 2}$, the restriction of $\sigma_t^\phi$ on $A$ coincides with the modular automorphism group on $A$ associated with $\phi$. This implies $\sigma^\phi_t(A)=A$ so that we have a natural inclusion $A\rtimes_{\sigma^\phi} \mathbb{R}\subset M\rtimes_{\sigma^\phi} \mathbb{R}$. 
Since $A$ is abelian (and so $\phi$ is tracial), $\sigma^\phi_t=\mathrm{id}_A$ on $A$ and hence we have $A\rtimes_{\sigma^\phi} \mathbb{R}=A\otimes L\mathbb{R}$. Then it is known that for any Tr-finite projection $p\in L\mathbb{R}$, the reduced subalgebra $A\otimes pL\mathbb{R}p$ is a Cartan subalgebra of a finite von Neumann algebra $p(M\rtimes_{\sigma^\phi} \mathbb{R})p$ (e.g. $\cite[\rm Propositions\ 2.6 \ and \ 2.7]{HR}$ and $\cite[\rm Lemma\ 2.2]{FSW10}$).

\subsection{\bf Popa's intertwining techniques}

As explained in Introduction, Sorin Popa introduced a useful tool which gives a good sufficient condition for unitary conjugacy of Cartan subalgebras. Here we recall only the precise statement which we need later. See $\cite[\rm Theorem\ F.12]{BO}$ for another proof.

\begin{Thm}[{\cite{Po 1}\cite{Po 2}}]\label{em thm}
Let $M$ be a finite von Neumann algebra with separable predual, $\tau$ a faithful normal trace on $M$, and let $A, B\subset M$ be (possibly non-unital) von Neumann subalgebras. We denote by $E_B$ the unique $\tau_B$-preserving conditional expectation from $1_BM1_B$ onto $B$, where $\tau_B:=\tau(1_B\cdot1_B)/\tau(1_B)$. Then the following conditions are equivalent.
\begin{itemize}
	\item[$(1)$]There exists no sequences $(w_n)_n$ of unitaries in $A$ such that  $\displaystyle\lim_n\|E_{B}(b^{\ast}w_na)\|_{2,\tau_B}=0$ for any $a,b\in 1_AM1_B$.
	\item[$(2)$]There exists a non-zero $A$-$B$-submodule $H$ of $1_AL^{2}(M,\tau)1_B$ such that $\dim_{(B,\tau_B)}H<\infty$.
	\item[$(3)$]There exist non-zero projections $e\in A$ and $f\in B$, a unital normal $\ast$-homomorphism $\theta \colon eAe \rightarrow fBf$, and a partial isometry $v\in M$ such that
\begin{itemize}
\item[$\bullet$] $vv^{\ast}\leq e$ and $v^{\ast}v\leq f$,
\item[$\bullet$] $v\theta(x)=xv$ for any $x\in eAe$.
\end{itemize}
\end{itemize}
We write as $A\preceq_{M}B$ if one of these conditions holds.
\end{Thm}

In the statement of (2) above, we do not need to take $L^2(M,\tau)$. 
In fact, we can choose any standard representation of $M$, since all standard representations are canonically isomorphic with each other including the left and right actions of $M$ (and hence those of $A$ and $B$).

\subsection{\bf Weak amenability and $\rm \bf W^*$CBAP}\label{wk am}

Weak amenability is an approximation property for discrete groups (more generally, locally compact groups) weaker than amenability, and the $\rm W^*$CBAP is a corresponding notion for von Neumann algebras.

To introduce these notions, we first recall the definition of a Herz--Shur multiplier. Let $\Gamma$ be a discrete group and $\phi$ a map from $\Gamma$ to $\mathbb{C}$. Consider a linear map 
\begin{eqnarray*}
m_\phi\colon \mathbb{C}[\Gamma]\rightarrow\mathbb{C}[\Gamma];\ \sum_{s\in \Gamma}a_s\cdot s \mapsto \sum_{s\in \Gamma}\phi(s)a_s\cdot s.
\end{eqnarray*}
Then regarding $\mathbb{C}[\Gamma]\subset C_\lambda^*(\Gamma)$, we define the Herz--Shur norm of $\phi$ as $\|\phi\|_{\rm c.b.}:=\|m_\phi\|_{\rm c.b.}$ (possibly infinite). We say $\phi$ (or $m_\phi$) is a \textit{Herz--Shur multiplier} if $\|\phi\|_{\rm c.b.}$ is finite.

Then Recall that a discrete group $\Gamma$ is \textit{weakly amenable} if there exists a net $(\phi_i)_i$ of finitely supported Herz--Shur multipliers satisfying $\limsup_i\|\phi_i\|_{\rm c.b.}<\infty$ and $\phi_i(g)\rightarrow 1$ as $i\rightarrow \infty$ for any $g\in \Gamma$.
We also recall that a von Neumann algebra $M$ has the $\it weak^*$ \textit{completely approximation property} (or $\ W^*$\textit{CBAP}, in short) if there exists a net $(\psi_i)_i$ of normal c.b.\ maps on $M$ with finite rank such that $\limsup_i\|\psi_i\|_{\rm c.b.}<\infty$ and $\psi_i$ converges to $\mathrm{id}_M$ in the point $\sigma$-weak topology.

Then optimal constants
\begin{eqnarray*}
\Lambda_{\rm c.b.}(\Gamma):=\inf\{\ \limsup_i\|\phi_i\|_{\rm c.b.}\mid (\phi_i) \textrm{ satisfies the above condition}\}\\
\Lambda_{\rm c.b.}(M):=\inf\{\ \limsup_i\|\psi_i\|_{\rm c.b.}\mid (\psi_i) \textrm{ satisfies the above condition}\}
\end{eqnarray*}
are invariants of $\Gamma$ and $M$ respectively, both of which are called the \textit{Cowling--Haagerup constant}. 
It is known that $\Lambda_{\rm c.b.}(\Gamma)=\Lambda_{\rm c.b.}(L\Gamma)$ (see for example $\cite[\rm Section\ 12.3]{BO}$).
Recently combined with an approximation property result of Brannan $\cite{Br12}$, Freslon proved that $\Lambda_{\rm c.b.}(L^\infty(\mathbb{G}))=1$, where $\mathbb{G}$ is monoidally equivalent to $A_o(1_n)$ or $A_u(1_n)$ $\cite{Fr12}$. However the general case is still open.

We will use these properties in two ways: one is Theorem $\ref{net}$ to get weakly compact actions; the other is as follows with local reflexivity. Recall a $C^*$-algebra $A$ is \textit{locally reflexive} if for any finite dimensional subspace $E\subset A^{**}$, there exists a net $(\mu_j)_j$ of c.c.\ maps from $E$ to A such that $(\mu_j)_j$ converges to $\mathrm{id}_E$ in the point $\sigma$-weak topology.
\begin{Lem}\label{ao w}
Let $M$ be a von Neumann algebra and $A\subset M$ a $\sigma$-weakly dense $C^*$-subalgebra. Let $(\phi_i)_i$ be a net of normal c.b.\ maps on $M$ with finite rank such that $\limsup_i\|\phi_i\|_{\rm c.b.}=:k<\infty$ and $\phi_i$ converges ${\rm id}_M$ in the point $\sigma$-weak topology. Assume $A$ is locally reflexive. Then we can find a net $(\psi_j)_j$ of normal c.b.\ maps from $M$ into $A$ with finite rank satisfying the same conditions as $(\phi_i)_i$.
\end{Lem}
\begin{proof}
Let $z\in A^{**}$ be the central projection satisfying $M\simeq zA^{**}$.
Put $E_i:=\phi_i(M)$ and regard as a subset of $A^{**}$ via $E_i\subset M\simeq zA^{**}$. Then, by local reflexivity of $A$, we can find a net $(\mu^i_j)_j$ of c.c.\ maps from $E_i$ into $A$ such that $\mu^i_j$ converges ${\rm id}_{E_{i}}$ in the point $\sigma$-weak topology. 
Now, putting $\tilde{\mu}^i_j(a):=z\mu^i_j(a)$, we have a net $(\tilde{\mu}^i_j\circ\phi_i)_{i,j}$ of c.b.\ maps from $M$ into $zA$ and this makes our desired net by using the identification $(A\subset M) \simeq (zA\subset zA^{**})$. 
\end{proof}
\subsection{\bf Universal discrete quantum groups}\label{uni dis}
In the paper, we use the quantum group theory only for Propositions $\ref{quan ao+}$ and $\ref{quan aoc+}$. We accept all the the basics of compact and discrete quantum groups and we refer the reader to $\cite{Wo}$ and $\cite{MV}$ for the details. Our notations are very similar to those of $\cite{VVe}$.

Let $C(\mathbb{G})$ be a compact quantum group. We denote by $\Phi$ the comultiplication, by $h$ the Haar state, and by $L^2(\mathbb{G})$ the GNS-representation of $h$. Then the Hilbert space $L^2(\mathbb{G})$ can be decomposed as follows:
\begin{equation*}
L^2(\mathbb{G})=\sum_{x\in\mathrm{Irred}(\mathbb{G})}\oplus (H_x\otimes H_{\bar{x}}),
\end{equation*}
where $\mathrm{Irred}(\mathbb{G})$ is the set of equivalent classes of all irreducible unitary representations of $\mathbb{G}$ and $\bar{x}$ is the contragredient of $x$. Let $t_x$ be the unique unit vector (up to multiplication by $\mathbb{T}$) in $H_x\otimes H_{\bar{x}}$ such that $(U^x\boxtimes U^{\bar{x}})$-invariant, where $U^x$ is the unitary element corresponding to $x$. Identify $t_x$ as an anti-linear map from $H_{\bar{x}}$ to $H_x$ with the Hilbert--Schmidt correspondence. Then we have two representations
\begin{alignat*}{5}
\rho&\colon& C(\mathbb{G})\rightarrow \mathbb{B}(L^2(\mathbb{G})) &;&\ \rho(\omega_{\eta,\xi}\otimes\iota(U^x))\Omega=\xi\otimes t^{\bar{x}}\eta\in H_x\otimes H_{\bar{x}},\\
\lambda&\colon& C(\mathbb{G})\rightarrow \mathbb{B}(L^2(\mathbb{G})) &;&\ \lambda(\omega_{\eta,\xi}\otimes\iota(U^x))\Omega=t^{\bar{x}}\eta\otimes\xi \in H_{\bar{x}}\otimes H_x,
\end{alignat*}
for all $x\in \mathrm{Irred}(\mathbb{G})$ and $\xi, \eta\in H_x$. Here $\Omega$ is the canonical cyclic vector. We note that these representations are unitarily equivalent to the GNS-representation for the Haar state $h$.
Define the dual discrete quantum group as 
\begin{alignat*}{5}
c_0(\hat{\mathbb{G}})&:=&\bigoplus_{x\in\mathrm{Irred}(\mathbb{G})}\mathbb{B}(H_x)&,&\\
\ell^\infty(\hat{\mathbb{G}})&:=&\prod_{x\in\mathrm{Irred}(\mathbb{G})}\mathbb{B}(H_x)&,&
\end{alignat*}
and define two representations of them on the same Hilbert space $L^2(\mathbb{G})$ by
\begin{alignat*}{5}
\hat{\lambda}&\colon&\ell^\infty(\hat{\mathbb{G}})\rightarrow \prod_{x\in\mathrm{Irred}(\mathbb{G})}\mathbb{B}(H_x)\otimes\mathbb{C}\subset\mathbb{B}(L^2(\mathbb{G}))&,&\\
\hat{\rho}&\colon&\ell^\infty(\hat{\mathbb{G}})\rightarrow \prod_{\bar{x}\in\mathrm{Irred}(\mathbb{G})}\mathbb{C}\otimes\mathbb{B}(H_x)\subset\mathbb{B}(L^2(\mathbb{G}))&.&
\end{alignat*}
All dual objects are written with hat (e.g. $\hat{\Phi}$, $\hat{h}$). 
{We have a natural unitary} 
\begin{equation*}
\mathbb{V}=\bigoplus_{x\in\mathrm{Irred}\mathbb{G}} U^x.
\end{equation*}


From now on, we assume that the Haar state $h$ is faithful on $C(\mathbb{G})$ and  recall modular objects of them. We use similar notations to $\cite{To}$ which has a good survey of the modular theory on compact quantum groups. Let $A(\mathbb{G})$ be the dense Hopf $*$-algebra of $C(\mathbb{G})$, $\kappa$ the antipode, and let $\epsilon$ be the counit of $C(\mathbb{G})$. 
Let $\{f_z\}_z$ ($z\in \mathbb{C}$) be the Woronowicz characters on $C(\mathbb{G})$, that is, a family of homomorphisms from $A(\mathbb{G})$ to $\mathbb{C}$ satisfying  conditions in $\cite[\rm Theorem\ 1.4]{Wo}$. Put $F_x:=(\iota\otimes f_1)(U^x)$ for $x\in\mathrm{Irred}(\mathbb{G})$. Then we have the following useful relations to the modular group associated with the Haar state $h$:
\begin{itemize}
\item $(\iota\otimes\sigma_t^{h})(U^x)=(F_x^{it}\otimes1)U^x(F_x^{it}\otimes1)\quad (t\in\mathbb{R}, x\in\mathrm{Irred}(\mathbb{G}))$,
\item $\displaystyle\Delta^{it}=\sum_{x\in\mathrm{Irred}(\mathbb{G})}\oplus (F_x^{it}\otimes F_{\bar{x}}^{-it})$ \quad on $\displaystyle L^2(\mathbb{G})=\sum_{x\in\mathrm{Irred}(\mathbb{G})}\oplus (H_x\otimes H_{\bar{x}})$.
\end{itemize}
We denote the scaling automorphism group by $\tau_t$ and the unitary antipode by $R$. Define a conjugate unitary $\hat{J}$ on $L^2(\mathbb{G})$ by $\hat{J}x\hat{1}:=R(x^*)\hat{1}$ for $x\in C(\mathbb{G})$ and put $U:=J\hat{J}=\hat{J}J$. 
Then we can identify all compact quantum group $C^*$-algebras as these opposite algebras.

Next we recall universal discrete quantum groups introduced in $\cite{VW96}$ which are our main objects. Let $F$ be an element in $\mathrm{GL}(n,\mathbb{C})$ $(n\geq 2)$. Then the  $C^*$-algebra $C(A_u(F))$ (respectively, $C(A_o(F))$ for $F\bar{F}=\pm1$) is defined as the universal unital $C^*$-algebra generated by all the entries of a unitary $n$ by $n$ matrix $u=(u_{i,j})_{i,j}$ satisfying 
\begin{itemize}
\item $F\bar{u}F^{-1}$ is unitary, (respectively, $F\bar{u}F^{-1}=u$).
\end{itemize}
where $\bar{u}=(u_{i,j}^*)_{i,j}$. 
Following $\cite{VVe}$ and $\cite{VV08}$, we treat only the case $n\geq3$. 

Put $\mathbb{G}:=A_o(F)$ or $A_u(F)$. Then $C(\mathbb{G})_{\rm red}$ is defined as an image of $C(\mathbb{G})$ in $\mathbb{B}(L^2(\mathbb{G}))$ via the GNS-representation and it is still a compact quantum group (with the haar state faithful). Write $L^\infty(\mathbb{G}):=C(\mathbb{G})_{\rm red}''$. Since previous two representations $\lambda$ and $\rho$ are unitarily equivalent, we naturally have 
\begin{equation*}
C(\mathbb{G})_{\rm red}\simeq\lambda(C(\mathbb{G}))\simeq\rho(C(\mathbb{G})),\quad
L^\infty(\mathbb{G})\simeq\lambda(C(\mathbb{G}))''\simeq\rho(C(\mathbb{G}))''.
\end{equation*} 
We regard $\rho(C(\mathbb{G}))\subset\rho(C(\mathbb{G}))''\subset\mathbb{B}(L^2(\mathbb{G}))$ as our main objects and, in the next section, we will prove that they satisfy some conditions similar to condition (AO). We note that factoriality and these types were studied in $\cite{Ba97}$ and $\cite{VVe}$ (but not solved completely).

All the irreducible representations of $A_o(F)$ and $A_u(F)$ were completely classified in the following sense $\cite{Ba96}\cite{Ba97}$: $\mathrm{Irred}(A_o(F))$ is identified with $\mathbb{N}$ in such a way that
\begin{equation*}
x\otimes y\simeq |x-y|\oplus(|x-y|+2)\oplus\cdots\oplus(x+y) \quad(x,y\in \mathbb{N});
\end{equation*}
$\mathrm{Irred}(A_u(F))$ is identified with $\mathbb{N}*\mathbb{N}$ in such a way that
\begin{equation*}
x\otimes y\simeq \bigoplus_{z\in \mathbb{N}*\mathbb{N}, x=x_0z, {y=\bar{z}y_0}} x_0y_0 \quad(x,y\in \mathbb{N}*\mathbb{N}).
\end{equation*}
From now on, for simplicity, we treat only $A_o(F)$ and all the cases of $A_u(F)$ in this paper can be treated in the same way as that of $A_o(F)$ (see $\cite[\rm Section\ 5]{VV08}$). 

Let $z$ be any irreducible representation contained in $x\otimes y$ as a subrepresentation (write as $z\in x\otimes y$). Let $p^{x\otimes y}_z$ be the unique projection in $\mathbb{B}(H_x\otimes H_y)$ satisfying $(U^x \boxtimes U^y)(p^{x\otimes y}_z\otimes 1)\simeq U^z$.
Then take an intertwiner $V(x\otimes y,z)$ between $(U^x\boxtimes U^y)(p^{x\otimes y}_z\otimes 1)$ and $U^z$ and it is unique up to multiplication by $\mathbb{T}$. Define a u.c.p.\ map $\psi_{x+y,x}\colon \mathbb{B}(H_x)\rightarrow\mathbb{B}(H_{x+y})$ by
\begin{equation*}
\psi_{x+y,x}(A):=V(x\otimes y,x+y)^*(A\otimes1)V(x\otimes y,x+y),
\end{equation*}
and note that this map does not depend on the choice of $V(x\otimes y,x+y)$. 

Finally we recall a nuclear $C^*$-subalgebra $\cal B$ of $\ell^\infty(\hat{\mathbb{G}})$ which plays a significant role for us. We first put 
\begin{equation*}
{\cal B}_0:=\{a\in \ell^\infty(\hat{\mathbb{G}})\mid \textrm{there exists $x$ such that }ap_y=\psi_{y,x}(ap_x) \textrm{ for all } y\geq x \}.
\end{equation*}
Let $\pi\colon \mathbb{B}(L^2(\mathbb{G}))\rightarrow\mathbb{B}(L^2(\mathbb{G}))/\mathbb{K}(L^2(\mathbb{G}))$ be the quotient map.
In $\cite{VVe}$, Vaes and Vergnioux proved that 
\begin{itemize}
	\item the norm closure ${\cal B}$ of ${\cal B}_0$ is a $C^*$-algebra containing $c_0(\hat{\mathbb{G}})$ so that the $C^*$-algebra ${\cal B}_\infty:={\cal B}/c_0(\hat{\mathbb{G}})$ is defined;
	\item $\cal B$ and ${\cal B}_\infty$ are nuclear;
	\item $\hat{\Phi}$ induces a left action of $\hat{\mathbb{G}}$ on $\cal B$ and ${\cal B}_\infty$;
	\item this left action on ${\cal B}_\infty$ is amenable so that ${\cal B}_\infty\rtimes_r\hat{\mathbb{G}}$ is nuclear and ${\cal B}_\infty\rtimes_r\hat{\mathbb{G}}={\cal B}_\infty\rtimes_{\rm full}\hat{\mathbb{G}}$;
	\item $\hat{\Phi}$ induces the trivial right action of $\hat{\mathbb{G}}$ on ${\cal B}_\infty$ so that ${\hat{\lambda}({\cal B}_\infty)}$ commutes with $\pi\circ\lambda(C(\mathbb{G}))$, where we identify ${\hat{\lambda}}$ as a map from $\ell^\infty(\mathbb{G})/c_0(\mathbb{G})$ to $\mathbb{B}(L^2(\mathbb{G}))/\mathbb{K}(L^2(\mathbb{G}))$.
\end{itemize}
Since {$\hat{\lambda}\colon{\cal B}_\infty\rightarrow \mathbb{B}(L^2(\mathbb{G}))/\mathbb{K}(L^2(\mathbb{G}))$ and $(\iota \otimes \pi\circ\rho)(\mathbb{V})\in M (c_0(\hat{\mathbb{G}})\otimes \mathbb{B}/\mathbb{K})$}
are a covariant representation for the left action $\hat{\Phi}$, we have the following $*$-homomorphism
\begin{equation*}
{\pi_l (=\hat{\lambda}\rtimes \pi\circ\rho)}\colon {\cal B}_\infty\rtimes_r\hat{\mathbb{G}}={\cal B}_\infty\rtimes_{\rm full}\hat{\mathbb{G}}\longrightarrow \mathbb{B}(L^2(\mathbb{G}))/\mathbb{K}(L^2(\mathbb{G}))
\end{equation*}
by universality. Putting $\pi_r:=\mathrm{Ad}U\circ\pi_l$, where $U=J\hat{J}$, we have the following algebraic $*$-homomorphism
\begin{alignat*}{5}
\pi_l\times\pi_r\colon &({\cal B}_\infty\rtimes_r\hat{\mathbb{G}})\odot ({\cal B}_\infty\rtimes_r\hat{\mathbb{G}})&\longrightarrow& \mathbb{B}(L^2(\mathbb{G}))/\mathbb{K}(L^2(\mathbb{G}))&\\
&\hspace{4.2em}a\otimes b&\longmapsto &\hspace{2.4em}\pi_l(a)\pi_r(b),&
\end{alignat*}
since ${\hat{\lambda}({\cal B}_\infty)}$ commutes with $\pi\circ\lambda(C(\mathbb{G}))$. Here $\odot$ means the algebraic tensor product. 
By nuclearity of ${\cal B}_\infty\rtimes_r\hat{\mathbb{G}}$, this map is min-bounded and the restriction of the map on $(\mathbb{C}\rtimes_r\hat{\mathbb{G}})\otimes (\mathbb{C}\rtimes_r\hat{\mathbb{G}})\simeq C(\mathbb{G})_{\rm red}\otimes C(\mathbb{G})_{\rm red}$ gives the min-boundedness of the multiplication map on $C(\mathbb{G})_{\rm red}$ after taking the quotient with $\mathbb{K}(L^2(\mathbb{G}))$. This is the proof of the fact that $L^\infty(\mathbb{G})$ satisfies condition (AO) given in $\cite{VVe}$. 

We note that the multiplication map from $C(\mathbb{G})_{\rm red}\otimes C(\mathbb{G})_{\rm red}$ to  $\mathbb{B}(L^2(\mathbb{G}))/\mathbb{K}(L^2(\mathbb{G}))$ is nuclear, since so is $({\cal B}_\infty\rtimes_r\hat{\mathbb{G}})\otimes ({\cal B}_\infty\rtimes_r\hat{\mathbb{G}})$ (and hence is $\pi_l\times\pi_r$). We will use this observation in the next section.

We finally mention that condition (AO) for $L^\infty(A_o(F))$ and $L^\infty(A_u(F))$ were first observed by Vergnioux in $\cite{Ve05}$.  

\section{\bf Conditions Similar to Ozawa's Condition (AO)}\label{sim ao}

In this section, we introduce some similar conditions to condition (AO). We will prove that von Neumann algebras of $A_o(F)$ and $A_u(F)$ satisfy these conditions.

\subsection{\bf Condition $\rm {\bf (AO)}^+$}

Let us first recall Ozawa's condition (AO). 
We say a von Neumann algebra $M\subset \mathbb{B}(H)$ satisfies $\it condition$ (AO) if there exist $\sigma$-weakly dense unital $C^*$-subalgebras $A\subset M$ and $B\subset M'$ such that 
\begin{itemize}
\item[$(\rm{i})$] $A$ is locally reflexive;
\item[$(\rm{ii})$] the multiplication map $\nu\colon A\odot B \rightarrow \mathbb{B}(H)/\mathbb{K}(H);\ a\otimes b\mapsto ab+\mathbb{K}(H)$ is min-bounded.
\end{itemize}

In $\cite{solid 1}$, Ozawa proved his celebrated theorem: if a finite von Neumann algebra satisfies condition (AO), then it is solid. As we mentioned, solidity (or semisolidity) implies primeness for non-injective $\rm II_1$ factors. 

The most important examples of von Neumann algebras with condition (AO) comes from bi-exact groups $\cite[\rm Definition\ 15.1.2]{BO}$. In fact, Ozawa proved that they have the following characterization $\cite[\rm Lemma\ 15.1.4]{BO}$: a countable discrete group $\Gamma$ is bi-exact if and only if $\Gamma$ is exact and satisfies the following condition
\begin{itemize} 
\item  there exists a u.c.p.\ map $\theta\colon C_{\lambda}^*(\Gamma)\otimes C_{\rho}^*(\Gamma) \rightarrow \mathbb{B}(\ell^2(\Gamma))$ such that $\theta(a\otimes b)-ab\in \mathbb{K}(\ell^2(\Gamma))$ for any $a \in C_{\lambda}^*(\Gamma)$ and  $b \in C_{\rho}^*(\Gamma)$.
\end{itemize}
It is now obvious that the group von Neumann algebras of bi-exact groups satisfy condition (AO). Thus he proved that factors of bi-exact non-amenable i.c.c.\ groups are solid, in particular, prime.

Here is another significant view point. To see solidity, we do \textit{not} need the existence of a u.c.p. map $\theta$ above. We need only the property that the multiplication map $\nu$ is min-bounded after taking the quotient with $\mathbb{K}(\ell^2(\Gamma))$. This is why condition (AO) is weaker than bi-exactness for group von Neumann algebras.

On the other hand, in $\cite{PV12}$, Popa and Vaes proved that the group von Neumann algebras of bi-exact and weakly amenable groups are strongly solid. In the proof, they used such a u.c.p.\ map $\theta$ as an essential tool. 

Motivated these observation, we define the first condition similar to condition (AO) as follows. 
\begin{Def}\label{ao+}\upshape
Let $M\subset \mathbb{B}(H)$ be a von Neumann algebra with standard representation and denote by $J$ the modular conjugation. We say $M\subset \mathbb{B}(H)$ satisfies $\it condition$ $\rm {(AO)}^+$ if there exists a unital $\sigma$-weakly dense $C^*$-subalgebra $A$ such that 
\begin{itemize}
	\item[$(\rm{i})$] $A$ is locally reflexive;
	\item[$(\rm{ii})$] there exists a u.c.p.\ map $\theta\colon A\otimes JAJ \rightarrow \mathbb{B}(H)$ such that $\theta(a\otimes JbJ)-aJbJ\in \mathbb{K}(H)$ for any $a,b \in A$.
\end{itemize}
\end{Def}
\bigskip

The difference of conditions (AO) and $\rm (AO)^+$ is of course the existence of a u.c.p.\ map $\theta$.  So it may be useful to consider how we get such a $\theta$ for von Neumann algebras satisfying condition (AO). For this purpose, we translate the second condition as follows. 
\begin{itemize}
\item[($\rm ii'$)] The multiplication map $\nu$ is min-bounded and it has a u.c.p.\ lift, that is, there exists a u.c.p.\ map $\theta\colon A\otimes JAJ \rightarrow \mathbb{B}(H)$ such that $\nu=\pi\circ\theta$, where $\pi\colon \mathbb{B}(H)\rightarrow\mathbb{B}(H)/\mathbb{K}(H)$ is the quotient map.
\end{itemize}
With this trivial translation, we can apply lifting theorems in some concrete cases. For example, if $A$ is separable $C^*$-algebra and the multiplication map $\nu$ is nuclear, then $\nu$ has a u.c.p.\ lift by the lifting theorem due to Choi and Effros $\cite{CE}$. This method has been used by Ozawa (see the proof of $\cite[\rm Proposition\ 15.2.3]{BO}$).

Now combined with the observation in Subsection $\ref{uni dis}$, we can easily deduce that our main targets satisfy condition $\rm {(AO)}^+$.
\begin{Pro}\label{quan ao+}
Von Neumann algebras $L^\infty(A_o(F))$ and $L^\infty(A_u(F))$ for $F\in \mathrm{GL}(n,\mathbb{C})$ $(n\geq 3)$ satisfy Condition $\rm {(AO)}^+$.
\end{Pro}

\subsection{\bf A similar condition for continuous cores}\label{cond core}
To see strong solidity in the rigidity theory, finiteness assumption is essential since all the known proofs require the theory of amenable trace, which works only for finite von Neumann algebras. However our main targets $L^\infty(A_o(F))$ and $L^\infty(A_u(F))$ are hardly finite. So it is natural for us to see the continuous cores of such factors which are always semifinite. 

In this subsection, we investigate some conditions for continuous cores of general von Neumann algebras. The following condition is a natural analogue of condition $\rm (AO)^+$ for continuous cores.

\begin{Def}\label{aoc+}\upshape
Let $M$ be a von Neumann algebra, $\phi$ a faithful normal state on $M$, and let $\tilde{J}$ be the modular conjugation for $M\rtimes_{\sigma^\phi}\mathbb{R}\subset \mathbb{B}(L^2(M,\phi)\otimes L^2(\mathbb{R}))$. We say the pair $(M,\phi)$ satisfies \textit{condition $\rm {(AO)}^+$ with respect to its continuous core} (say \textit{condition $\rm (AOC)^+$}, in short) if there exists a $\sigma$-weakly dense unital $C^*$-subalgebra $A\subset M$ such that 
\begin{itemize}
	\item[$(\rm i)$] $\sigma^{\phi}$ defines a norm continuous action on $A$ (so that we can define $A\rtimes_r\mathbb{R}$);
	\item[$(\rm ii)$] $A\rtimes_r\mathbb{R}$ is locally reflexive;
	\item[$(\rm iii)$] there exists a u.c.p.\ map 
\begin{equation*}
\theta\colon (A\rtimes_r\mathbb{R})\odot \tilde{J}(A\rtimes_r\mathbb{R})\tilde{J} \longrightarrow \mathbb{B}(L^2(M,\phi)\otimes L^2(\mathbb{R}))
\end{equation*}
such that $\theta(a\otimes \tilde{J}b\tilde{J})-a\tilde{J}b\tilde{J}\in \mathbb{K}(L^2(M,\phi))\otimes\mathbb{B}(L^2(\mathbb{R}))$ for any $a,b\in A\rtimes_r\mathbb{R}$.
\end{itemize}
\end{Def}

Our goal in the subsection is to show that $A_o(F)$ and $A_u(F)$ with the Haar states satisfy this condition. For this, we investigate a sufficient condition for condition $\rm {(AOC)}^+$.

Let $M$ be a von Neumann algebra and $\phi$ a faithful normal state on $M$. Write $H:=L^2(M,\phi)$ and ${\cal K}:=\mathbb{K}(H)\otimes\mathbb{B}(L^2(\mathbb{R}))$  and let $J$ be the modular conjugation on $H$. Consider the multiplier algebra ${\cal L}:=M({\cal K})$ of ${\cal K}$ and denote 
${\cal C}:={\cal L/K}.$

Assume first that there exists a $\sigma$-weakly dense unital $C^*$-subalgebra $A\subset M$ such that 
\begin{itemize}
	\item[(a)] $\sigma^{\phi}$ defines a norm continuous action on $A$ (so that we can define $A\rtimes_r\mathbb{R}$).
\end{itemize}

Let $\pi$ be a $*$-homomorphism from $\mathbb{B}(H)$ into $\mathbb{B}(H\otimes L^2(\mathbb{R}))$ given by $(\pi(x)\xi)(t):=\Delta_\phi^{-it}x\Delta_\phi^{it}\xi(t)$ for $x\in \mathbb{B}(H)$ and $t\in \mathbb{R}$. Consider the $C^*$-algebra $D$ generated by
following elements
\begin{itemize}
\item $\pi(a)$, $JbJ\otimes1$ $(a,b\in A)$;
\item $1\otimes\lambda_t$, $\Delta_\phi^{it}\otimes\rho_{t}$ $(t\in\mathbb{R})$;
\item $\int_\mathbb{R} f(s)(1\otimes\lambda_s)\cdot ds$, $\int_\mathbb{R} f(s)(\Delta_\phi^{is}\otimes\rho_{s})\cdot ds$ $(f\in L^1(\mathbb{R}))$.
\end{itemize}
Then we assume that 
\begin{itemize}
	\item[(b)] $D$ is contained in $\cal L$.
\end{itemize}
In particular, we have natural maps from $A$ and $A^{\rm op}(=JAJ)$ to $\cal C$. We denote these maps by $\pi_l$ and $\pi_r$ respectively.

Next we assume that
\begin{itemize}
	\item[(c)] there exist separable nuclear $C^*$-algebras $C_l$ and $C_r$ containing $A$ and $A^{\rm op}$ respectively (so that $A$ is exact);
	 \item[(d)] there exist $*$-homomorphisms from $C_l$ and $C_r$ to $\cal C$ such that they are extensions of $\pi_l$ and $\pi_r$, respectively. We still denote them by $\pi_l$ and $\pi_r$.
\end{itemize}
Then we want to define the following $*$-homomorphism
\begin{alignat*}{5}
\nu\colon C_l\odot C_r \longrightarrow {\cal C} ;\ a\otimes b\mapsto \pi_l(a)\pi_r(b).
\end{alignat*}
However we do not know whether ranges of $C_l$ and $C_r$ commute, and hence we further assume that 
\begin{itemize}
\item[(e)] $\nu$ is a well-defined $*$-homomorphism, that is, $[\pi_l(a),\pi_r(b)]=0$  ($a\in C_l$, $b\in C_r$).
\end{itemize}
We can extend $\nu$ on $C_l\otimes C_r$ by the nuclearity. 
Restricting this map, we have a natural multiplication $*$-homomorphism
\begin{alignat*}{5}
\nu\colon A\otimes A^{\rm op} \longrightarrow {\cal C}\ ;\ 
&a\otimes1\ &\longmapsto& \ [\pi(a)],\\
&1\otimes a^{\rm op}\ &\longmapsto& \ [a^{\rm op}\otimes1].
\end{alignat*}
Next consider norm continuous $(\mathbb{R}\times\mathbb{R})$-actions on $A\otimes A^{\rm op}$ and $\nu(A\otimes A^{\rm op})$ given by
\begin{alignat*}{5}
&\mathbb{R}\times\mathbb{R} \longrightarrow \mathrm{Aut}(A\otimes A^{\rm op})\ &;&\ 
&s\otimes t\ &\longmapsto& \ \sigma_s^{\phi}\otimes\tilde{\sigma}_t^{\phi},\qquad\qquad\\
&\mathbb{R}\times\mathbb{R} \longrightarrow \mathrm{Aut}(\nu(A\otimes A^{\rm op}))\ &;&\ 
&s\otimes t\ &\longmapsto& \ \mathrm{Ad} ([1\otimes\lambda_s][\Delta_{\phi}^{it}\otimes\rho_t]).
\end{alignat*}
Here $\tilde{\sigma}_t^\phi(a^{\rm op})=\tilde{\sigma}_t^\phi(Ja^*J):=J\sigma_t^\phi (a^*)J=\sigma_t^\phi(a)^{\rm op}$. It is easily verified that $\nu$ is $(\mathbb{R}\times\mathbb{R})$-equivariant and hence we have the following $*$-homomorphism:
\begin{equation*}
\tilde{\nu}\colon A\rtimes_r\mathbb{R}\otimes \tilde{J}(A\rtimes_r\mathbb{R})\tilde{J}\simeq(A\otimes A^{\rm op})\rtimes_r(\mathbb{R}\times\mathbb{R}) \rightarrow (\nu(A\otimes A^{\rm op}))\rtimes_r(\mathbb{R}\times\mathbb{R})\rightarrow {\cal C}.
\end{equation*}
Here the continuity of the final map comes from the amenability of $\mathbb{R}\times\mathbb{R}$.
The resulting map says that the multiplication map on $A\rtimes_r\mathbb{R}\odot (A\rtimes_r\mathbb{R})^{\rm op}$ to ${\cal L}\subset\mathbb{B}(H\otimes L^2(\mathbb{R}))$ is min-bounded after taking the quotient with ${\cal K}$. Now the $C^*$-algebra $A\rtimes_r\mathbb{R}$ is exact (and hence locally reflexive) since so is $A$, and it is $\sigma$-weakly dense in $M\rtimes_{\sigma^\phi}\mathbb{R}$. At this point, $M\rtimes_{\sigma^\phi}\mathbb{R}$ satisfies a similar condition to condition (AO). 

Finally we assume that 
\begin{itemize}
\item[(f)] $\mathrm{Ad} ([1\otimes\lambda_s][\Delta_{\phi}^{it}\otimes\rho_t])\nu(C_l\otimes C_r)=\nu(C_l\otimes C_r)$ for any $(s,t)\in\mathbb{R}\times\mathbb{R}$ and this defines a norm continuous action on $\nu(C_l\otimes C_r)$.
\end{itemize}
In this case there exists a $*$-homomorphism from $\nu(C_l\otimes C_r)\rtimes_r(\mathbb{R}\times\mathbb{R})$ into $\cal C$ and hence the image of the map is nuclear. Since $\mathrm{ran} \tilde{\nu}$ is contained in this image, $\tilde{\nu}$ is a nuclear map into $\cal C$. Thus the lifting theorem of Choi and Effros is again applicable so that $\tilde{\nu}$ has a u.c.p.\ lift. Summary we have the following lemma.
\begin{Lem}
Let $M$ be a von Neumann algebra, $\phi$ a faithful normal state on $M$, and let $A\subset M$ be a $\sigma$-weakly dense unital $C^*$-subalgebra. If they satisfy all the conditions from $\rm(a)$ to $\rm(f)$, then $(M,\phi)$ satisfies condition $\rm (AOC)^+$.
\end{Lem}

Now we turn to show our main objects satisfy these conditions.

\begin{Pro}\label{quan aoc+}
Von Neumann algebras $L^\infty(A_o(F))$ and $L^\infty(A_u(F))$ for $F\in \mathrm{GL}(n,\mathbb{C})$ $(n\geq 3)$ with the Haar state $h$ satisfy condition $\rm (AOC)^+$.
\end{Pro}
\begin{proof}
We keep the notations in Subsection $\ref{uni dis}$. Put $A:=C^*_{\rm red}(\mathbb{G})=\rho(C(\mathbb{G}))\subset \mathbb{B}(L^2(\mathbb{G}))$ and $C_l=C_r={\cal B}_\infty\rtimes_r\mathbb{G}$. We will verify all the conditions from (a) to (f) above. Note that the condition (a) is a well-know property.

For this, recall the following formula: for any irreducible decomposition $x\otimes y\simeq \sum_{z\in x\otimes y}\oplus z$, we have 
\begin{equation*}
F_x\otimes F_y\simeq\sum_{z\in x\otimes y}\oplus F_z \quad \textrm{on } H_x\otimes H_y\simeq \sum_{z\in x\otimes y}\oplus H_z.
\end{equation*}
Indeed this follows from a direct calculation of $(\iota\otimes\iota\otimes f_1)(U^x_{13}U^y_{23})$. By the formula, we have the following relation:
\begin{equation*}
\Delta^{it}\hat{\lambda}(\psi_{x+y,x}(B))\Delta^{-it}=\hat{\lambda}(\psi_{x+y,x}(F_{x}^{it}BF_x^{-it})) \quad (B\in \mathbb{B}(H_x)).
\end{equation*}
In this sense, the modular group $\Delta^{it}$ commutes with all $\psi_{x+y,y}$. 

To see the condition (b), observe first that for any $c\in \mathbb{B}(H\otimes L^2(\mathbb{R}))$, $c$ is contained in $ {\cal K}$ if and only if 
\begin{equation*}
\|\sum_{x=0}^{y}(p_x\otimes 1)c\sum_{x=0}^{z}(p_x\otimes 1) -c\|\rightarrow 0 \quad (y,z\rightarrow \infty).
\end{equation*} 
Let $a$ be an element in ${\cal K}$ and $b$ a generator of $D$. We will show $ba\in  {\cal K}$. 
The cases $b=1\otimes\lambda_t$, $\Delta_\phi^{it}\otimes\rho_{t}$, and $d\otimes 1$ $(d\in UAU^*)$ are trivial. 
The cases $b=\int_\mathbb{R} f(s)(1\otimes\lambda_s)\cdot ds$ and $\int_\mathbb{R} f(s)(\Delta_\phi^{is}\otimes\rho_{s})\cdot ds$ $(f\in L^1(\mathbb{R}))$ are easy since they commute with all $p_x$. 
For the final case $b=\pi(d)$ $(d\in A)$, we may assume $d=(\omega_{\xi,\eta}\otimes\rho)(\mathbb{V})$ for $\xi, \eta\in H_z$ and $z\in \mathrm{Irred}(\mathbb{G})$. Let $(\xi_k^z)_{k=1}^{n_z}$ be a fixed orthonormal basis of $H_z$.
Then we have
\begin{eqnarray*}
p_x\sigma_{t}(d)
&=&p_x\sigma_{t}((\omega_{\xi,\eta}\otimes\rho)(\mathbb{V}))\\
&=&(\omega_{\xi,\eta}\otimes\iota)((1\otimes p_x)\mathscr{V}_t) \qquad (\mathscr{V}_t:=(\iota\otimes\sigma_t\circ\rho)(\mathbb{V}))\\
&=&(\omega_{\xi,\eta}\otimes\iota)(\mathscr{V}_t\mathscr{V}_t^*(1\otimes p_x)\mathscr{V}_t)\\
&=&\sum_{k=1}^{n_z}(\omega_{\xi,\xi_k^z}\otimes\iota)(\mathscr{V}_t)(\omega_{\xi_k^z,\eta}\otimes\iota)((1\otimes \Delta^{it})\hat{\Psi}(p_x)(1\otimes \Delta^{-it}))
\end{eqnarray*}
for all $x\in\mathrm{Irred}(\mathbb{G})$, where $\sigma_t$ is the modular group for the Haar state and $\hat{\Psi}(p_x):=(\iota\otimes \rho)(\mathbb{V})^*(1\otimes p_x)(\iota\otimes \rho)(\mathbb{V})$. Since $\sum_{x=0}^yp_x$ converges to $1$ in the strong topology as $y\rightarrow \infty$ and $\hat{\Psi}$ is normal, 
each $(\omega_{\xi_k^z,\eta}\otimes\iota)((1\otimes \Delta^{it})\hat{\Psi}(\sum_{x=0}^yp_x)(1\otimes \Delta^{-it}))$ converges to $\omega_{\xi_k^z,\eta}(1)1$ in the strong topology. Hence for any compact operator $T\in \mathbb{K}(H)$, the equation
\begin{equation*}
\sum_{x=0}^yp_x\sigma_{t}(d)T=\sum_{k=1}^{n_k}(\omega_{\xi,\xi_k^z}\otimes\iota)(\mathscr{V}_t)(\omega_{\xi_k^z,\eta}\otimes\iota)((1\otimes \Delta^{it})\hat{\Psi}(\sum_{x=0}^yp_x)(1\otimes \Delta^{-it}))T
\end{equation*}
implies that $\sum_{x=0}^yp_x\sigma_{t}(d)T$ converges to $\sigma_{t}(d)T$ in the \textit{norm} topology as $y\rightarrow \infty$. We choose $T$ as an element of the set of all linear combinations of the form $\xi_k^a\otimes \xi_l^b$ for $a,b\in \mathrm{Irred}\mathbb{G}$ and $k,l$.
Then it is easy to verify that this convergence is uniform with respect to $t\in\mathbb{R}$, that is, 
\begin{equation*}
\sup_{t\in\mathbb{R}}\|(1-\sum_{x=0}^yp_x)\sigma_{t}(d)T\|\rightarrow 0 \qquad (y\rightarrow \infty).
\end{equation*}
Now, for any $S\in \mathbb{B}(L^2(\mathbb{R}))$, we have 
\begin{eqnarray*}
&&\|\sum_{x=0}^y(p_x\otimes1)\pi(d)(T\otimes S)-\pi(d)(T\otimes S)\|\\
&=&\|\int_\mathbb{R}(\sum_{x=0}^yp_x-1)\sigma_{-t}(d)\otimes e_t\cdot dt(T\otimes S)\|\\
&\leq&\sup_{t\in\mathbb{R}}\|(\sum_{x=0}^yp_x-1)\sigma_{t}(d)T\|\|S\|\rightarrow 0 \quad(y\rightarrow \infty).
\end{eqnarray*}
Hence, for any $a\in {\cal K}=\mathbb{K}(H)\otimes \mathbb{B}(L^2(\mathbb{R}))$, we have 
\begin{equation*}
\|\sum_{x=0}^y(p_x\otimes1)\pi(d)a-\pi(d)a\|\rightarrow 0 \quad(y\rightarrow \infty).
\end{equation*}
Since $\pi(d)a\sum_{x=0}^y(p_x\otimes1)$ converges to $\pi(d)a$ in the norm topology, $\pi(d)a$ is contained in $\cal K$, and we get the condition (b).

Next we define two maps $\pi_l$ and $\pi_r$. We begin with following maps
\begin{alignat*}{5}
{\cal B} \longrightarrow {\cal L}\ ;\ 
&a\ &\longmapsto& \ \pi(\hat{\lambda}(a)),\\
{\cal B} \longrightarrow {\cal L}\ ;\ &a\ &\longmapsto& \ U\hat{\lambda}(a)U^*\otimes1.
\end{alignat*}
It is not difficult to see that ranges of these maps are really contained in ${\cal L}$. Since images of $c_0(\hat{\mathbb{G}})$ by these maps are contained in $\cal K$, we have induced maps from ${\cal B}_\infty$ to ${\cal C}$. Simple calculations show that these maps make two covariant representations of ${\cal B}_\infty$ and the natural left action of $\hat{\mathbb{G}}$. Since this action is amenable we have following desired maps: 
\begin{alignat*}{5}
\pi_l&\colon& &C_l={\cal B}_\infty\rtimes_r \hat{\mathbb{G}}& \longrightarrow {\cal C},&\\
\pi_r&\colon& &C_r={\cal B}_\infty\rtimes_r \hat{\mathbb{G}}& \longrightarrow {\cal C}.&
\end{alignat*}

Finally we prove the condition (e) (and then the condition (f) is easily verified). For this, it suffices to see that $\pi\circ\hat{\lambda}({\cal B})$ (respectively, $\hat{\rho}({\cal B})\otimes 1$) commutes with $\lambda(C(\mathbb{G}))\otimes 1$ (respectively, $\pi\circ\rho(C(\mathbb{G}))$) after taking the quotient with $\cal K$. Here we treat only the case of $\pi\circ\hat{\lambda}({\cal B})$ and $\lambda(C(\mathbb{G}))\otimes 1$, and the other case follows from the same manner.

Let $z$ be an element of $\mathrm{Irred}(\mathbb{G})$ and write as $U^z=\sum_{i,j}u_{i,j}^z\otimes e_{i,j}$, where $(e_{ij})_{ij}$ is a fixed matrix unit in $\mathbb{B}(H_z)$. Our goal is to show 
\begin{equation*}
[\pi\circ\hat{\lambda}(b), \lambda({u_{i,j}^z}^*)\otimes 1]\in \mathbb{K}(L^2(\mathbb{G}))\otimes \mathbb{B}(L^2(\mathbb{R}))
\end{equation*}
for any $z,i,j$ and any $b\in {\cal B}$, where $[\cdot,\cdot]$ is the commutator. Since this term coincides with 
\begin{equation*}
\int_{\mathbb{R}}[\Delta^{-it}\hat{\lambda}(b)\Delta^{it},\lambda({u_{i,j}^z}^*)]\otimes e_t\cdot dt,
\end{equation*}
running over all $i$ and $j$, our goal is equivalent to 
\begin{equation*}
\int_{\mathbb{R}}\sum_{i,j}[\Delta^{-it}\hat{\lambda}(b)\Delta^{it},\lambda({u_{i,j}^z}^*)]\otimes e_{i,j}\otimes e_t\cdot dt \in \mathbb{K}(L^2(\mathbb{G}))\otimes\mathbb{B}(H_z)\otimes \mathbb{B}(L^2(\mathbb{R})),
\end{equation*}
and using $(\lambda\otimes\hat{\lambda})(\mathbb{V}_{21}^*)(1\otimes p_z)=\sum_{i,j}{u_{i,j}^z}^*\otimes e_{i,j}\otimes \mathrm{id}_{H_{\bar{z}}}$ (write $W:=(\lambda\otimes\hat{\lambda})(\mathbb{V}_{21})$), we further translate it as
\begin{equation*}
\int_{\mathbb{R}}[\Delta^{-it}\hat{\lambda}(b)\Delta^{it}\otimes p_z,W^*(1\otimes p_z)]\otimes e_t\cdot dt \in \mathbb{K}(L^2(\mathbb{G}))\otimes\mathbb{B}(H_z\otimes H_{\bar{z}})\otimes \mathbb{B}(L^2(\mathbb{R})).
\end{equation*}
For simplicity, we denote it by $\int_{\mathbb{R}}{\rm T}_{t}\otimes e_t\cdot dt$. 
If $b\in{\cal B}$ is finitely supported, that is, contained in a finite direct sum of $\mathbb{B}(H_x)$ ($x\in \mathrm{Irred}(\mathbb{G})$), then this final condition holds since this term is contained in $(\oplus_{\rm fin}\mathbb{B}(H_x\otimes H_{\bar{x}}))\otimes\mathbb{B}(H_z\otimes H_{\bar{z}})\otimes \mathbb{B}(L^2(\mathbb{R})).$ Hence we may assume that $b=\psi_{\infty, x}(A)\in {\cal B}$ for some $A\in\mathbb{B}(H_x)$, where $\psi_{\infty, x}(A)$ is defined by $\psi_{\infty, x}(A)p_y:=\psi_{y,x}(A)$ if $y\geq x$ and 0 otherwise.

Now by the proof of $\cite[\rm Proposition\ 3.8]{VVe}$, for any $y\in\mathrm{Irred}(\mathbb{G})$ with $y\geq z$ we have 
\begin{eqnarray*}
&&\| {\rm T}_t(p_{x+y}\otimes p_z)\|\\
&=&\|[\Delta^{-it}\hat{\lambda}(\psi_{\infty, x}(A))\Delta^{it}\otimes 1,W^*](p_{x+y}\otimes p_z)\|\\
&=&\|[\hat{\lambda}(\psi_{\infty, x}(F_x^{-it}AF_x^{it}))\otimes 1,W^*](p_{x+y}\otimes p_z)\| \\
&=&\|\{W(\hat{\lambda}(\psi_{\infty, x}(F_x^{-it}AF_x^{it}))\otimes 1)W^*-\hat{\lambda}(\psi_{\infty, x}(F_x^{-it}AF_x^{it}))\otimes 1\}(p_{x+y}\otimes p_z)\|\\
&=&\|(\hat{\lambda}\otimes\hat{\lambda})\{\hat{\Phi}(\psi_{\infty, x}(F_x^{-it}AF_x^{it})\otimes 1)-\psi_{\infty, x}(F_x^{-it}AF_x^{it})\otimes 1\}(p_{x+y}\otimes p_z)\|\\
&\leq&C(z)\|F_x^{-it}AF_x^{it}\|q^y=C(z)\|A\|q^{y},
\end{eqnarray*}
where $C(z)$ and $0<q<1$ are constants ($C(z)$ depends on $z$). Since this estimate does not depend on $t\in\mathbb{R}$, we have the following norm convergent sequence 
\begin{equation*}
\int_{\mathbb{R}}\sum_{k=0}^{y}({\rm T}_{t}(p_{x+k}\otimes p_z))\otimes e_t\cdot dt \rightarrow  \int_{\mathbb{R}}\sum_{k=0}^{\infty}({\rm T}_{t}(p_{x+k}\otimes p_z))\otimes e_t\cdot dt \quad (y\rightarrow \infty).
\end{equation*}
Now each element in this sequence is contained in $\mathbb{K}(L^2(\mathbb{G}))\otimes\mathbb{B}(H_z\otimes H_{\bar{z}})\otimes \mathbb{B}(L^2(\mathbb{R}))$ and the limit element coincides with $\int_{\mathbb{R}}{\rm T}_{t}\otimes e_t\cdot dt$. Hence we can end the proof.
\end{proof}

\section{\bf Absence of Cartan subalgebras}

In this section, we prove Theorem \ref{A} and B in the almost same way as $\cite[\rm Theorem\ 3.1]{PV12}$. Since many proofs are same, we often omit them.

\subsection{\bf Preparation with the $\bf W^*$CBAP and condition $\rm \bf (AO)^+$}

Since we have similar arguments for proofs of both theorems, we first assume that $M$ is arbitrary semifinite von Neumann algebra with separable predual, and we will give other assumptions in each lemma.

Let $M$ be a semifinite von Neumann algebra with separable predual, Tr a faithful normal semifinite trace, $p$ a Tr-finite projection in $M$, and let $A\subset pMp$ be a von Neumann subalgebra.  Write $P:={\cal N}_{pMp}(A)''$. For simplicity we assume $\mathrm{Tr}(p)=1$.
As usual, we identify $J_MMJ_M$ and $J_PPJ_P$ as opposite algebras $M^{\rm op}$ and $P^{\rm op}$ via natural identifications and write $\bar{a}:=(a^{\rm op})^*(=J_MaJ_M$ or $J_PaJ_P$) for $a\in M$ or $P$. 
Put $D:=M\odot M^{\rm op}\odot P^{\rm op}\odot P$ and define two $*$-homomorphisms 
\begin{alignat*}{5}
\Psi&\colon D\longrightarrow& &\mathbb{B}(L^2(M)\otimes L^2(M)\otimes L^2(P))&;&\ 
a\otimes b^{\rm op}\otimes x^{\rm op}\otimes y\longmapsto a\otimes b^{\rm op}\otimes x^{\rm op}y,\\
\Theta&\colon D\longrightarrow& &\mathbb{B}(L^2(M)\otimes L^2(P))&;&\ 
a\otimes b^{\rm op}\otimes x^{\rm op}\otimes y\longmapsto ab^{\rm op}\otimes x^{\rm op}y.
\end{alignat*}

The following theorem is due to Popa and Ozawa. In the theorem, $L^2(A)$, $L^2(P)$, and $L^2(pMp)$ means GNS representations of $\mathrm{Tr}(p\cdot p)$.

\begin{Thm}[{\cite{OP 07}\cite{Oz 10}}]\label{net}
If $pMp$ has the $W^*CBAP$ and $A$ is injective, then the conjugate action of ${\cal N}_{pMp}(A)$ on $A$ is weakly compact, that is, there exists a net $(\xi_i)_i$ of unit vectors in the positive cone of $L^2(A)\otimes L^2(A)\subset L^2(pMp)\otimes L^2(P)$ satisfying the following conditions:
\begin{itemize}
	\item[$\rm(i)$] $\langle (x\otimes1)\xi_i, \xi_i\rangle\rightarrow \mathrm{Tr}(x)$ for any $x\in pMp$;
	\item[$\rm(ii)$] $\|(a\otimes \bar{a})\xi_i-\xi_i\|\rightarrow 0$ for any $a\in {\cal U}(A)$;
	\item[$\rm(iii)$] $\|\xi_i-(u\otimes \bar{u})J(u\otimes \bar{u})J)\xi_i\|\rightarrow 0$ for any $u\in {\cal N}_{pMp}(A)$, where $J:=J_{pMp}\otimes J_P$.
\end{itemize}
\end{Thm}
\bigskip

Note that regarding $(\xi_i)_i$ as vectors in $L^2(M,\mathrm{Tr})\otimes L^2(P)$ by $L^2(pMp)\simeq pJ_MpJ_M L^2(M,\mathrm{Tr})\subset L^2(M,\mathrm{Tr})$, we get 
\begin{itemize}
	\item[$\rm(i')$] $\langle (x\otimes1)\xi_i, \xi_i\rangle\rightarrow \mathrm{Tr}(pxp)$ for any $x\in M$.
\end{itemize}
In our proof, we will use $(\xi_i)_i$ as vectors in the positive cone of $L^2(M)\otimes L^2(P)$ satisfying $(pJ_MpJ_M\otimes 1_P) \xi_i=\xi_i$ and conditions $\rm (i)'$, (ii), and (iii). 
In this case, we can exchange $L^2(M)$ and $L^2(P)$ with any other standard representations of $M$ and $P$, since these conditions are preserved under adjoint maps by canonical unitaries between standard representations.

We fix such a net $(\xi_i)_i$ and put $\Omega_1(x):={\rm Lim}_i\langle x\xi_i,\xi_i\rangle$ for $x\in \mathbb{B}(L^2(M)\otimes L^2(P))$, where Lim is taken by a fixed free ultra filter. 
Then conditions $\rm (i')$, (ii) and (iii) are translated as following conditions:
\begin{itemize}
	\item[$\rm(iv)$] $\Omega_1(x\otimes 1)= \mathrm{Tr}(pxp)$ for any $x\in M$;
	\item[$\rm(v)$] $\Omega_1(a\otimes \bar{a})=1$ for any $a\in {\cal U}(A)$;
	\item[$\rm(vi)$] $\Omega_1(\Theta(u\otimes \bar{u}\otimes \bar{u}\otimes u))=1$ for any $u\in {\cal N}_{pMp}(A)$.
\end{itemize}

We next prove the following lemma. In the original paper of Popa and Vaes, the proof of the corresponding statement is very technical (whose origin is in $\cite[\rm Lemma\ 6.2]{CSU11}$). In the present paper, we give a very simple proof which works only in our special situation. We are indebted to Eric Ricard for kindly demonstrating to the author this simple proof.

\begin{Lem}\label{net2}
Let $\Omega$ be a state on $\mathbb{B}(L^2(M)\otimes L^2(P))$ satisfying condition $\rm (v)$ above. 
Assume that $A$ is diffuse.
Then $\Omega(x\otimes 1)=0$ for any $x\in \mathbb{K}(L^2(M))$. In particular,
 there exists an increasing net $(p_j)_j$ of range finite projections in $\mathbb{B}(L^2(M))$ such that
\begin{itemize}
\item $p_j\rightarrow 1_M$ strongly;
\item $\Omega(p_j\otimes1_P)=0$ for any $p_j$.
\end{itemize}
\end{Lem}
\begin{proof}
Since $A$ is diffuse, we can find a sequence $(u_n)$ of unitaries in $A$ satisfying the following conditions:
\begin{itemize}
	\item $(u_n)_n$ converges to 0 in the $\sigma$-weak topology;
	\item $(\frac{1}{n} \sum_{k=1}^n u_k p u^{*}_k)_n$ converges to 0 in the \textit{norm} topology for any $p\in \mathbb{K}(L^2(B))$.
\end{itemize}
By condition (v), ${\cal U}(A)$ is contained in the multiplicative domain of $\Omega$ (e.g.\ $\cite[\textrm{Proposition 1.5.7}]{BO}$). 
Since each $u_n$ is a unitary in $A$,
 we have, for any $p\in \mathbb{K}(L^2(M))$
\begin{eqnarray*}
\Omega(p\otimes 1_P)&=&\Omega(u_kpu_k^* \otimes 1_P)\\
&=&\Omega(\frac{1}{n} \sum_{k=1}^n u_k p u^{*}_k \otimes 1_P) \rightarrow 0.
\end{eqnarray*} 
\end{proof}

Next, we assume that $M$ satisfies condition $\rm (AO)^+$ with a $\sigma$-weakly dense $C^*$-subalgebra  $M_0$ and a u.c.p.\ map $\theta$. Put $D_0:=M_0\odot M_0^{\rm op}\odot P^{\rm op}\odot P \subset D$. 
We use the map $\theta$ only in the following simple lemma. 


\begin{Lem}\label{ao net}
Let $p_j$ be range finite projections in $\mathbb{B}(L^2(M))$ with $p_j\rightarrow1$ strongly and assume $M$ satisfies condition $\rm (AO)^+$. Then we have $\limsup_j\|\Theta(S)(p_j^{\perp}\otimes1)\|\leq\|\Psi(S)\|$ for any $S\in D_0$, where $p_j^\perp:=1-p_j$.
\end{Lem}
\begin{proof}
For $S:=a\otimes b^{\rm op}\otimes x^{\rm op}\otimes y\in D_0$, we have 
\begin{equation*}
(\theta\otimes{\rm id})\circ\Psi(S)-\Theta(S)=(\theta(a\otimes b^{\rm op})-ab^{\rm op})\otimes x^{\rm op}y\in\mathbb{K}(L^2(M))\otimes\mathbb{B}(L^2(P)).
\end{equation*}
Since  $p_j^{\perp}$ converges to 0 in the strong topology and $\theta(a\otimes b^{\rm op})-ab^{\rm op}\in\mathbb{K}(L^2(M))$, the net $(\theta(a\otimes b^{\rm op})-ab^{\rm op})p_j^{\perp}$ converges to 0 in the \textit{norm} topology.
Hence we have
\begin{eqnarray*}
\limsup_j\|\Theta(S)(p_j^{\perp}\otimes1)\|
&\leq&\limsup_j\|((\theta\otimes{\rm id})\circ\Psi(S)-\Theta(S))(p_j^{\perp}\otimes1)\|+\|\Psi(S)\|\\
&=&\|\Psi(S)\|.
\end{eqnarray*}
This holds for any $S\in D_0$ by the completely same manner.
\end{proof}

\subsection{\bf Proof of Theorem \ref{A}}
We prove the following theorem which is slightly general than Theorem \ref{A}.

\begin{Thm}\label{thmA}
Let $M$ be a semifinite von Neumann algebra with separable predual and $p$ a finite projection in $M$. If $M$ satisfies condition $\rm (AO)^+$ and has the $\it W^*CBAP$, then $pMp$ is strongly solid.
\end{Thm}
\begin{proof}
Let $A\subset pMp$ be a diffuse injective von Neumann subalgebra. 
Take $(\xi_i)_i$, $\Omega_1$ $D$, $D_0$, $\Psi$, and $\Theta$ as in the previous subsection for the pair $A\subset pMp$. Take any increasing net $(p_j)_j$ of finite rank projections in $\mathbb{B}(L^2(M))$. 
Under this setting, we can completely follow the proof of $\cite[\rm Subsection\ 3.4]{PV12}$. Here we give a sketch of the proof for reader's convenience.

By Lemma $\ref{net2}$ and $\ref{ao net}$, we have for any $S\in D_0$
\begin{equation*}
|\Omega_1(\Theta(S))|=\limsup_j|\Omega_1(\Theta(S)(p_j^\perp\otimes1))|\leq\limsup_j\|\Theta(S)(p_j^\perp\otimes1)\|\leq\|\Psi(S)\|.
\end{equation*}
We can extend this inequality on $D$ (up to a scalar multiple) by condition (iv) of $\Omega_1$ and the $\rm W^*CBAP$ of $M$ (use Lemma $\ref{ao w}$ if necessary). Then a positive functional $\Omega_2$ on $C^*\{\Psi(D)\}$ is defined by $\Omega_2(\Psi(X)):=\Omega_1(\Theta(X))$. This is a state since $\Omega_2(1)=1$. 
Take a Hahn--Banach extension of $\Omega_2$ on $\mathbb{B}(L^2(M)\otimes L^2(M)\otimes L^2(P))$. 
Thanks for conditions (iv) and (vi) of $\Omega_1$, the restriction of this extended state on $p\mathbb{B}(L^2(M))p\otimes \mathbb{C}1_M\otimes \mathbb{C}1_P$ is an $P$-central state which restricts $\mathrm{Tr}(p\cdot p)$ on $pMp$. 
Hence $P$ is injective.
\end{proof}

\subsection{\bf Proof of Theorem \ref{B}}

We next prove Theorem \ref{B} with a very similar argument. We actually prove the following statement.

\begin{Thm}\label{thmB}
Let $M$ be a von Neumann algebra with separable predual and $\phi$ a faithful normal state on $M$. Let $N\subset M$ be a diffuse non-injective von Neumann subalgebra with a faithful normal conditional expectation $E_N$ which preserves $\phi$. If $(M,\phi)$ satisfies condition $\rm (AOC)^{+}$ and has the $\it W^*CBAP$, then $N$ has no {$\phi$-Cartan subalgebras}.
\end{Thm}
\begin{proof}
Suppose by contradiction that $N$ has a {$\phi$-Cartan subalgebra} $B$ with a $\phi$-preserving conditional expectation $E_B$.  
Taking crossed products by $\mathbb{R}$ with the modular action of $\phi$, we have inclusions $B\otimes L\mathbb{R}\subset N\rtimes\mathbb{R}:={\cal N}\subset M\rtimes\mathbb{R}:={\cal M}$. Denote by Tr the canonical trace on $\cal M$. 
We can find a non zero projection $p$ in $L\mathbb{R}\subset{\cal M}$ with $\mathrm{Tr}(p)<\infty$ such that $p{\cal N}p$ is still non-injective. Write $A:=B\otimes pL\mathbb{R}p$ and $P:={\cal N}_{p{\cal M}p}(A)''$. 
Note that $P$ is non-injective since it contains the non-injective subalgebra ${\cal N}_{p{\cal N}p}(A)''=p{\cal N}p$ (see Subsection \ref{TT}).
We will apply an almost same argument as that in Theorem \ref{A} to the inclusion $A\subset p{\cal M}p$, and will get injectivity of $P$ which means a contradiction.

Let $(\xi_i)_i$, $\Omega_1$ $D$, $\Psi$, and $\Theta$ be as before for the pair $A\subset p{\cal M}p$ ($\cal M$ has {the $\rm W^*CBAP$}, see $\cite[4.10]{An95}$). We define $D_0:=M_0\rtimes_r \mathbb{R}\odot (M_0\rtimes_r \mathbb{R})^{\rm op}\odot P^{\rm op} \odot P$, where $M_0$ is a $\sigma$-weakly dense $C^*$-subalgebra of $M$ as in the definition of condition $\rm (AOC)^+$. 
As mentioned in the observation below Theorem $\ref{net}$, we can choose $\xi_i$ as vectors in $L^2({\cal M},\tilde{\phi})\otimes L^2(P)$. Recall $L^2({\cal M},\tilde{\phi})=L^2(M)\otimes L^2(\mathbb{R})$.
In the setting, by a similar argument to that in Lemma $\ref{net2}$ and $\ref{ao net}$, we can prove the following statements:
\begin{itemize}
	\item $\Omega_1(x\otimes 1_{L\mathbb{R}}\otimes 1_P)=0$ for any $x\in \mathbb{K}(L^2(M))$;
	\item for any range finite projections $p_j$ in $\mathbb{B}(L^2(M))$ with $p_j\rightarrow1$ strongly, we have $\limsup_j\|\Theta(S)(p_j^{\perp}\otimes1_{L\mathbb{R}}\otimes 1_P)\|\leq\|\Psi(S)\|$ for any $S\in D_0$.
\end{itemize}
Here we used diffuseness of $B$ (not of $A$) and condition $\rm (AOC)^+$ of $M$. 
Now we can completely follow the proof the Theorem \ref{A} to get injectivity of $P$. 
%
\end{proof}
\section{\bf Semisolidity of continuous cores}
In the section, we prove Theorem \ref{C}. Our proof is a variant of the proof of $\cite[\rm Theorem\ 4.6]{solid 3}$ and is very similar to that of $\cite[\rm Theorem\ 5.3.3]{Is}$.

\subsection{\bf Proof of Theorem \ref{C}}

For simplicity, we write the core of $M$ as ${\cal M}:=M\rtimes_{\sigma^\phi} \mathbb{R}$. We use $L^2({\cal M},\tilde{\phi})=L^2(M)\otimes L^2(\mathbb{R})$ as a representation (although $\tilde{\phi}$ may not be a trace). Let $N$ be a type $\rm II_1$ subalgebra of $p{\cal M}p$. Since $N$ contains a copy of the AFD $\rm II_1$ factor, we may assume that $N$ itself is the AFD $\rm II_1$ factor. 
Let $N_n\subset N$ $(n\in\mathbb{N})$ be an increasing finite dimensional unital $C^{\ast}$-subalgebras whose union is dense in $N$. We define a conditional expectation from $\mathbb{B}(L^2({\cal M}))$ onto $N'\cap p\mathbb{B}(L^2({\cal M}))p=N'p$ by
\begin{equation*}
\Psi_{N}(x):=\mathrm{Lim}_{n}\int_{{\cal U}(N_{n})}uxu^{\ast}du,
\end{equation*}
where $du$ is the normalized Haar measure on ${\cal U}(N_{n})$ and Lim is taken by a fixed ultrafilter. Then $\Psi_{N}$ satisfies a properness condition
\begin{equation*}
\Psi_{N}(x)\in \overline{\rm{co}}^w\{uxu^{\ast}\mid u\in {\cal U}(N)\} \quad (x\in\mathbb{B}(L^2({\cal M}))).
\end{equation*}
We only prove that $\mathbb{K}(L^2(M))\otimes\mathbb{B}(L^2(\mathbb{R}))\subset \ker \Psi_{N}$ and the theorem follows from same manners as that in $\cite[\rm Theorem\ 4.6]{solid 3}$ or $\cite[\rm Theorem\ 5.3.3]{Is}$.

To see this, It suffice to show that $\Psi_{N}(x\otimes1)=0$ for $x=J\hat{a}\otimes J\hat{b}\in \mathbb{K}(L^2(M))$, where $a,b\in M$  and we used the Hilbert--Schmidt correspondence $(\xi\otimes \eta)\zeta:= \langle \zeta, \eta\rangle \xi$.
Then since 
\begin{eqnarray*}
\Psi_N((J\hat{a}\otimes J\hat{b})\otimes1)
=(JaJ\otimes1)\Psi_N((\hat{1}\otimes \hat{1})\otimes1)(Jb^*J\otimes1),
\end{eqnarray*}
we may assume $a=b=1$. Write $e:=\hat{1}\otimes \hat{1}$, which is the orthogonal projection onto $\mathbb{C}\hat{1}$.
Let $q$ be any projection in $L\mathbb{R}$ with $\mathrm{Tr}(q)<\infty$. 
Put $\tilde{q}:=\tilde{J}q\tilde{J}\in {\cal M}'\subset N'$, $\widetilde{N}:=N\tilde{q}$ and $\Psi_{\widetilde{N}}(x):=\tilde{q}\Psi_N(x)\tilde{q}$ for $x\in \mathbb{B}(L^2({\cal M}))$. 
We actually prove $\Psi_{\widetilde{N}}(e\otimes 1)=0$ (this means $\Psi_N(e\otimes 1)=0$ by the choice of $q$).

Since $\Psi_N$ is proper, $\Psi_N(e\otimes1)$ commutes with $\Delta_{\phi}^{it}\otimes \rho_t$ and  so it is contained in $pN'\cap \rho(L\mathbb{R})'$, where $\rho(\lambda_t):=\Delta_{\phi}^{it}\otimes \rho_t$. 
Hence $\Psi_{\widetilde{N}}(e\otimes1)$ is contained in $\tilde{q}(pN'\cap \rho(L\mathbb{R})')\tilde{q}\subset N'\cap \rho(qL\mathbb{R}q)'$.
Let $r$ be any spectral projection of $\Psi_{\widetilde{N}}(e\otimes1)$, which corresponds to the interval $[\epsilon, \|\Psi_{\widetilde{N}}(e\otimes1)\|]$ for any small $\epsilon>0$. 
Then since $r$ is also contained in $N'\cap \rho(qL\mathbb{R}q)'$ and $r\leq p\tilde{q}$, $rL^2({\cal M})$ has a natural $N$-$qL\mathbb{R}q$-submodule structure of $p\tilde{q}L^2({\cal M})(\simeq p\tilde{q}L^2(f{\cal M}f)$, where $f:=p\vee q$). 
Since we know $N\not\preceq_{f{\cal M}f}L\mathbb{R}q$ (because $N$ is of type $\rm II_1$ and $L\mathbb{R}$ is of type I), by Theorem \ref{em thm} and the comment below it, the dimension of $rL^2({\cal M})$ with respect to $(L\mathbb{R}q,\mathrm{Tr}_{L\mathbb{R}q})$ is zero or infinite, where $\mathrm{Tr}_{L\mathbb{R}q}:=\mathrm{Tr}(\cdot )/\mathrm{Tr}(q)$ for the canonical trace Tr on ${\cal M}$.

Let $W$ be the unitary on $L^2({\cal M})=L^2(M)\otimes L^2(\mathbb{R})$ given by $(W\xi)(t):=\Delta_{\phi}^{it}\xi(t)$. Then easy calculations show that 
\begin{eqnarray*}
W(1\otimes \lambda_t)W^*=\Delta^{it}\otimes \lambda_t, \quad 
W(\Delta^{it}\otimes \rho_t)W^*=1\otimes \rho_t
\end{eqnarray*}
and hence we have 
\begin{eqnarray*}
\rho(qL\mathbb{R}q)' \cap \mathbb{B}(L^2({\cal M})q)
=\tilde{q}\rho(L\mathbb{R})' \tilde{q}
=W^*(\mathbb{B}(L^2(M))\otimes \bar{q}L\mathbb{R})W,
\end{eqnarray*}
where $\bar{q}:= J_{L\mathbb{R}}qJ_{L\mathbb{R}}$ (note $1\otimes\bar{q}=W\tilde{q}W^*$).
Then $\dim_{L\mathbb{R}q}rL^2({\cal M})$ coincides with that of the right $L\mathbb{R}q$-module $WrW^*(L^2(M)\otimes L^2(\mathbb{R})q)$, where the right action is given by $1\otimes \rho_t\bar{q}$ ($t\in \mathbb{R}$). 
From the fundamental theory of dimension, the dimension of $WrW^*(L^2(M)\otimes L^2(\mathbb{R})q)$ is smaller than $(\mathrm{Tr}_{L^2(M)}\otimes \mathrm{Tr})(WrW^*)/\mathrm{Tr}(q)$,  and this is finite since
\begin{eqnarray*}
(\mathrm{Tr}_{L^2(M)}\otimes\mathrm{Tr})(WrW^*)
&\leq&C\cdot(\mathrm{Tr}_{L^2(M)}\otimes\mathrm{Tr})(W\Psi_{\widetilde{N}}(e\otimes1)W^*)\\
&\leq&C\cdot(\mathrm{Tr}_{L^2(M)}\otimes\mathrm{Tr})(W\tilde{q}(e\otimes 1)\tilde{q}W^*)\\
&=&C\cdot\mathrm{Tr}_{L^2(M)}(e)\mathrm{Tr}(\bar{q})<\infty,
\end{eqnarray*}
where $C$ is a positive constant and we used properness of $\Psi_N$. Hence we have  $r=0$ and $\Psi_N(e\otimes1)=0$. Thus we proved the claim.


\subsection{\bf A remark on solidity and centralizer algebras}\label{so}

In the previous subsection, we proved semisolidity of continuous cores of some type $\rm III_1$ factors. This property itself has nothing to say about original type $\rm III_1$ factors at a first glance, but it has an interesting application once we get a stronger property, namely, solidity of the continuous cores. Indeed Houdayer gave the following observation $\cite[\rm Subsection\ 3.3]{Ho08}$.

Let $M$ be a type $\rm III_1$ factor and assume that the continuous core $M\rtimes \mathbb{R}$ is solid as a $\rm II_\infty$ factor. Let $\phi$ be a faithful normal state on $M$ and $M_\phi$ be the centralizer of $\phi$ (see for example $\cite[\rm Definition\ VIII.2.1]{Tak 2}$).
Then by Takesaki's conditional expectation theorem $\cite[\rm Theorem\ IX.4.2]{Tak 2}$, there exists the unique $\phi$-preserving conditional expectation $E$ from $M$ onto $M_\phi$. Hence by the observation in Subsection $\ref{TT}$, we have 
$M_\phi\otimes L\mathbb{R}=M_{\phi}\rtimes_{\sigma^\phi}\mathbb{R}\subset M\rtimes_{\sigma^\phi}\mathbb{R}.$ 
Then for any $\mathrm{Tr}$-finite projection $p\in L\mathbb{R}$, $M_\phi\otimes \mathbb{C}p$ is injective since it is contained in the relative commutant of $L\mathbb{R}p$ and $L\mathbb{R}$ is diffuse. 
Thus the solidity of the continuous core forces all the centralizers (with respect to states) to be injective. 

To apply this observation to our main objects, we next recall Connes' discrete decomposition of full type $\rm III$ factors $\cite[\rm Section\ 4]{Co74}$ (See also $\cite[\rm Section\ 2]{Dy94}$). 
\begin{Thm}
Let $M$ be a full type $\rm III$ factor with separable predual and $\phi$ a faithful normal state on $M$. Assume that $\phi$ is $\mathrm{Sd}(M)(=:\Gamma)$-almost periodic. Then there exists a decomposition $M\simeq (M \otimes \mathbb{B}(\ell^2(\Gamma)))_{\phi\otimes \omega}\rtimes\Gamma$ with a faithful normal semifinite weight $\omega$ on $\mathbb{B}(\ell^2(\Gamma))$. 
The algebra $(M \otimes \mathbb{B}(\ell^2(\Gamma)))_{\phi\otimes \omega}$ is a $\rm II_\infty$ factor and is isomorphic to $M_\phi\otimes \mathbb{B}(H)$ for some separable Hilbert space $H$.
\end{Thm}
By the crossed product decomposition in the theorem, since $M$ is non-injective and $\Gamma$ is amenable, $(M \otimes \mathbb{B}(\ell^2(\Gamma)))_{\phi\otimes \omega}$ is also non-injective. 
Hence we get non-injectivity of $M_\phi$.

We turn to see our main objects. Let $\mathbb{G}$ be a universal quantum group $A_o(F)$ for $F\in \mathrm{GL}(n,\mathbb{C})$ $(n\geq 3)$ with $F\bar{F}=\pm 1$. 
Assume that $\|F\|^2\leq \mathrm{Tr}_n(FF^*)/\sqrt{5}$, where $\mathrm{Tr}_n$ is the trace on $\mathbb{M}_n(\mathbb{C})$ with $\mathrm{Tr}_n(1)=n$. Then recall from $\cite[\rm Theorem\ 7.1]{VVe}$ that $L^\infty(\mathbb{G})$ then satisfies following conditions:
\begin{itemize}
\item the algebra $L^\infty(\mathbb{G})$ is a full factor and the Haar state $h$ is almost periodic;
\item the invariant $\mathrm{Sd}(L^\infty(\mathbb{G}))$ is the subgroup $\Gamma$ of $\mathbb{R}_+^*$ generated by eigenvalues of $Q\otimes Q^{-1}$, where $Q^{-1}:=FF^*$. In particular $L^\infty(\mathbb{G})$ is of type $\rm II_1$ if $FF^*=1$; of type $\rm III_\lambda$ $(0<\lambda<1)$ if $\Gamma=\lambda^\mathbb{Z}$; of type $\rm III_1$ in the other cases.
\end{itemize}
In the case, since the Haar state $h$ is $\mathrm{Sd}(L^\infty(\mathbb{G}))$-almost periodic, we have non-injectivity of $L^\infty(\mathbb{G})_h$ by the theorem above (in the $\rm II_1$ factor case $FF^*=1$, non-injectivity is trivial). Hence Houdayer's observation says that the continuous core of $L^\infty(\mathbb{G})$ is not solid. We summary this result as follows.
\begin{Cor}
Let $\mathbb{G}$ be a universal quantum group $A_o(F)$ for $F\in \mathrm{GL}(n,\mathbb{C})$ $(n\geq3)$ with $F\bar{F}=\pm1$. Assume that $\|F\|^2\leq \mathrm{Tr}_n(FF^*)/\sqrt{5}$. 
Denote the continuous core of $L^\infty(\mathbb{G})$ by $\cal M$ and the canonical trace on $\cal M$ by $\rm{Tr}$. Then for any $\rm{Tr}$-finite projection $p$ in ${\cal M}$ with $p{\cal M}p$ non-injective, $p{\cal M}p$ is semisolid but never solid.
\end{Cor}

\noindent 
{\bf Acknowledgement.} The author would like to thank Professors Yasuyuki Kawahigashi, who is his adviser, Eric Ricard, Reiji Tomatsu, Yoshimichi Ueda, and Makoto Yamashita for their valuable comments. In particular, he appreciate his colleague Yuki Arano for fruitful conversations on compact quantum groups. He was supported by Research Fellow of the Japan Society for the Promotion of Science.



\end{document}